\documentclass{amsart}

\usepackage{amsfonts,amssymb,verbatim,amsmath,amsthm,latexsym,textcomp,amscd}
\usepackage{latexsym,amsfonts,amssymb,epsfig,verbatim}
\usepackage{amsmath,amsthm,amssymb,latexsym,graphics,textcomp, hyperref}
\usepackage{graphicx}
\usepackage{color}
\usepackage{url}

\begin{document}

\newtheorem{theorem}{Theorem}[section]
\newtheorem{prop}[theorem]{Proposition}
\newtheorem{lemma}[theorem]{Lemma}
\newtheorem{cor}[theorem]{Corollary}
\newtheorem{definition}[theorem]{Definition}
\newtheorem{defn}[theorem]{Definition}
\newtheorem{conj}[theorem]{Conjecture}
\newtheorem{claim}[theorem]{Claim}
\newtheorem{defth}[theorem]{Definition-Theorem}
\newtheorem{obs}[theorem]{Observation}
\newtheorem{rmark}[theorem]{Remark}
\newtheorem{qn}[theorem]{Question}
\newtheorem{theo}[theorem]{Theorem}
\newtheorem{thmbis}{Theorem}
\newtheorem{dfn}[theorem]{Definition} 
\newtheorem{defi}[theorem]{Definition} 
\newtheorem{coro}[theorem]{Corollary}
\newtheorem{hypo}[theorem]{Hypothesis}
\newtheorem{corbis}{Corollary}
\newtheorem{propbis}{Proposition} 
\newtheorem*{prop*}{Proposition} 
\newtheorem{lem}[theorem]{Lemma} 
\newtheorem{lembis}{Lemma} 
\newtheorem{claimbis}{Claim} 
\newtheorem{fact}[theorem]{Fact} 
\newtheorem{factbis}{Fact} 
\newtheorem{qst}[theorem]{Question} 
\newtheorem{qstbis}{Question} 
\newtheorem{pb}[theorem]{Problem} 
\newtheorem{pbbis}{Problem} 
 \newtheorem{question}[theorem]{Question}
\newtheorem{rem}[theorem]{Remark}
\newtheorem{remark}[theorem]{Remark}
\newtheorem{example}[theorem]{Example}
\newtheorem{eg}[theorem]{Example}
\newtheorem{notation}[theorem]{Notation}
\newenvironment{preuve}[1][Preuve]{\begin{proof}[#1]}{\end{proof}}

\newcommand{\hhat}{\widehat}
\newcommand{\boundary}{\partial}
\newcommand{\C}{{\mathbb C}}
\newcommand{\integers}{{\mathbb Z}}
\newcommand{\natls}{{\mathbb N}}
\newcommand{\bbN}{{\mathbb N}}
\newcommand{\hypp}{{{\mathbb H}^2}}
\newcommand{\hyps}{{{\mathbb H}^3}}
\newcommand{\ratls}{{\mathbb Q}}
\newcommand{\reals}{{\mathbb R}}
\newcommand{\bbR}{{\mathbb R}}
\newcommand{\lhp}{{\mathbb L}}
\newcommand{\tube}{{\mathbb T}}
\newcommand{\cusp}{{\mathbb P}}
\newcommand\AAA{{\mathcal A}}
\newcommand\BB{{\mathcal B}}
\newcommand\CC{{\mathcal C}}
\newcommand\DD{{\mathcal D}}
\newcommand\EE{{\mathcal E}}
\newcommand\FF{{\mathcal F}}
\newcommand\GG{{\mathcal G}}
\newcommand\HH{{\mathcal H}}
\newcommand\II{{\mathcal I}}
\newcommand\JJ{{\mathcal J}}
\newcommand\KK{{\mathcal K}}
\newcommand\LL{{\mathcal L}}
\newcommand\MM{{\mathcal M}}
\newcommand\NN{{\mathcal N}}
\newcommand\OO{{\mathcal O}}
\newcommand\PP{{\mathcal P}}
\newcommand\QQ{{\mathcal Q}}
\newcommand\RR{{\mathcal R}}
\newcommand\SSS{{\mathcal S}}
\newcommand\TT{{\mathcal T}}
\newcommand\UU{{\mathcal U}}
\newcommand\VV{{\mathcal V}}
\newcommand\WW{{\mathcal W}}
\newcommand\XX{{\mathcal X}}
\newcommand\YY{{\mathcal Y}}
\newcommand\ZZ{{\mathcal Z}}
\newcommand\Ga{{\Gamma}}
\newcommand\pslc{{PSL_2({\C})}}
\newcommand\pslr{{PSL_2({\bbR})}}
\newcommand\CH{{\CC\Hyp}}
\newcommand\MF{{\MM\FF}}
\newcommand\PMF{{\PP\kern-2pt\MM\FF}}
\newcommand\ML{{\MM\LL}}
\newcommand\PML{{\PP\kern-2pt\MM\LL}}
\newcommand\GL{{\GG\LL}}
\newcommand\Pol{{\mathcal P}}
\newcommand\half{{\textstyle{\frac12}}}
\newcommand\Half{{\frac12}}
\newcommand\Mod{\operatorname{Mod}}
\newcommand\Area{\operatorname{Area}}
\newcommand\ep{\epsilon}
\newcommand\Hypat{\widehat}
\newcommand\Proj{{\mathbf P}}
\newcommand\U{{\mathbf U}}
 \newcommand\Hyp{{\mathbf H}}
 \newcommand\bH{{\mathbf H}}
\newcommand\D{{\mathbf D}}
\newcommand\Z{{\mathbb Z}}
\newcommand\R{{\mathbb R}}
\newcommand\Q{{\mathbb Q}}
\newcommand\E{{\mathbb E}}
\newcommand\til{\widetilde}
\newcommand\tK{\widetilde{K}}
\newcommand\btK{bdy(\widetilde{K})}
\newcommand\bK{bdy({K})}
\newcommand\tX{\widetilde{X}}
\newcommand\tP{\widetilde{P}}
\newcommand\tM{\widetilde{M}}
\newcommand\tS{\widetilde{S}}
\newcommand\length{\operatorname{length}}
\newcommand\tr{\operatorname{tr}}
\newcommand\gesim{\succ}
\newcommand\lesim{\prec}
\newcommand\simle{\lesim}
\newcommand\simge{\gesim}
\newcommand{\simmult}{\asymp}
\newcommand{\simadd}{\mathrel{\overset{\text{\tiny $+$}}{\sim}}}
\newcommand{\ssm}{\setminus}
\newcommand{\pair}[1]{\langle #1\rangle}
\newcommand{\T}{{\mathbf T}}
\newcommand{\inj}{\operatorname{inj}}
\newcommand{\collar}{\operatorname{\mathbf{collar}}}
\newcommand{\bcollar}{\operatorname{\overline{\mathbf{collar}}}}
\newcommand{\I}{{\mathbf I}}
\newcommand{\pis}{{\pi_1(S)}}

\newcommand{\bbar}{\overline}
\newcommand{\UML}{\operatorname{\UU\MM\LL}}
\newcommand{\EL}{\mathcal{EL}}
\newcommand\MT{{\mathbb T}}
\newcommand\Teich{{\mathcal T}}

\makeatletter
\@tfor\next:=abcdefghijklmnopqrstuvwxyzABCDEFGHIJKLMNOPQRSTUVWXYZ\do{%
  \def\command@factory#1{%
    \expandafter\def\csname cal#1\endcsname{\mathcal{#1}}
    \expandafter\def\csname frak#1\endcsname{\mathfrak{#1}}
    \expandafter\def\csname scr#1\endcsname{\mathscr{#1}}
    \expandafter\def\csname bb#1\endcsname{\mathbb{#1}}
    \expandafter\def\csname rm#1\endcsname{\mathrm{#1}}
  }
 \expandafter\command@factory\next
}
\makeatother

\newcommand*{\longhookrightarrow}{\ensuremath{\lhook\joinrel\relbar\joinrel\rightarrow}}
\newcommand {\tto}{ \to\rangle }
\newcommand {\onto} {\twoheadrightarrow}
\newcommand {\into} {\hookrightarrow}
\newcommand {\xra} {\xrightarrow}    
\newcommand{\ra}{\rightarrow}
\newcommand{\imp} {\Rightarrow}
\newcommand{\actedon}{\curvearrowleft} 
\newcommand{\actson}{\curvearrowright}

\newcommand {\sd} {\rtimes}   
\newcommand{\semidirect}{\ltimes}
\newcommand{\isemidirect}{\rtimes}
\newcommand{\tensor}{\otimes}
\newcommand{\wreath}{\Lbag}

\newcommand{\du}{\sqcup}
\newcommand{\Dunion}{\bigsqcup} 
\newcommand{\disjoint}{\sqcup}
\newcommand{\normal} {\vartriangleleft}

\newcommand {\ie}{ i.e.  }

\newcommand{\ul}[1]{\underline{#1}} 
\newcommand{\ol}[1]{\overline{#1}}


\newcommand{\Cay}{\operatorname{Cay}}
\newcommand{\proj}{\operatorname{proj}}
\newcommand{\Fix} {\operatorname{Fix}}
\newcommand{\dist}{\operatorname{dist}}
\newcommand{\diam}{\mathop{\mathrm{diam}\;}}

\newcommand{\yam}{$\bullet$}
\newcommand{\modif}{$ \clubsuit$}
\newcommand{\coucou}[1]{\footnote{#1}\marginpar{$\leftarrow$}}
\newcommand{\needref}{\textsuperscript{{\it citation needed}}\marginpar{$\leftarrow$}}
\newcommand{\why}{\textsuperscript{{\it Why ?}}\marginpar{$\leftarrow$}}

\title{Planes in degenerate 3-manifolds}

\author[Mahan Mj]{Mahan Mj}

\address{Tata Institute of Fundamental Research. 5, Homi Bhabha Road, Mumbai-400005, India}

\email{mahan@math.tifr.res.in}


\thanks{Research partially supported by  a DST J C Bose Fellowship. }   

\date{\today}

 \begin{abstract}
 We study totally geodesic planes in hyperbolic 3-manifolds $M$ having incompressible core and
degenerate ends.
We prove a Ratner-type phenomenon: a closed minimal
$\pslr-$invariant subset of $M$ is either an immersed totally geodesic
surface  or all of $M$. 

We also show that for an arbitrary  infinite volume hyperbolic 3-manifold $M$
without parabolics and with
 finitely generated fundamental group,  the number of compact totally geodesic surfaces in $M$ is finite.
\end{abstract}

\maketitle

\tableofcontents

\section{Introduction} 
The aim of this paper is to prove
the following Ratner-type phenomenon for degenerate hyperbolic 3-manifolds
(see Theorem \ref{main}):

\begin{theorem}\label{int}  Let $\Gamma$ be a degenerate Kleinian group without parabolics
and $M = \Hyp^3/\Gamma$ be the 
associated degenerate hyperbolic 3-manifold. Further, assume that the compact core $K$
 of $M$ is incompressible.
 Let $X$ be a minimal closed 
$H-$invariant subset, where $H=\pslr$. Then $X$ is either an immersed totally geodesic
surface or all of $M$.
\end{theorem}

Theorem \ref{int} has a long history going back to Hedlund's Theorem \cite{hedlund}. Work of
Margulis \cite{margulis-opp} and  Ratner's spectacular resolution of Raghunathan's conjecture 
\cite{ratner-top, ratner} gave a complete answer to the analogous problem for lattices in semi-simple Lie groups.
The special case in dimension three  \cite{shah} says that for a complete hyperbolic 3-manifold
of finite volume,
any closed 
$H-$invariant subset is either an immersed totally geodesic
surface  or all of the ambient manifold. Theorem \ref{int} is thus an exact analog of this Theorem
for degenerate hyperbolic 3-manifolds. 

The problem of classifying closed 
$K-$invariant subsets of a manifold $M$ whose geometry is modeled on a larger (semi-simple)
group $G$ makes sense
for any pair $(G,K)$ whenever $K \subset G$ and is of interest when $K$ is generated by
unipotents \cite{ratner-top}. However the problem 
has received considerably less attention for infinite volume manifolds.
 Recently,
McMullen, Mohammadi and Oh \cite{mmo} 
(resp. \cite{mmo-hor}) satisfactorily completed the study of closed 
$H-$invariant subsets of rigid acylindrical hyperbolic 3-manifolds when $H = \pslr$ 
(resp.  the group of real unipotent
matrices). They
 posed the problem for more general classes of hyperbolic 3-manifolds in 
 \cite{mmo-hor}. In this paper we devote ourselves to manifolds 
at the opposite end of the spectrum, viz. hyperbolic 3-manifolds all whose ends are degenerate. 
We note that this is the first departure from the convex cocompact
world.

In the final Section of the paper, we relax the assumption that $M$ has incompressible core.
 McMullen, Mohammadi and Oh \cite[Theorem B.1]{mmo} proved that for a convex cocompact Kleinian group $\Ga$, there
can be only finitely many compact
totally geodesic surfaces in $M = \hyps/\Ga$. In this paper, we relax the hypothesis of convex cocompactness
and  prove the following:
\begin{theorem} \label{int2} (see Theorem \ref{main2})
Let $\Ga$ be a finitely generated Kleinian group without parabolics and let
$M = \hyps/\Gamma$. If $M$ has infinite volume, then there can exist only finitely many compact
totally geodesic surfaces in $M$. \end{theorem}

It is curious that, in the present paper,
 we do not need the assumption of acylindricity that is essential in \cite{mmo}.
 Degeneracy of the ends seems to force
some kind of mixing for a totally geodesic plane.
As pointed out in \cite{mmo},  the
Ratner-type phenomenon discussed there and in the present paper fails for  certain cylindrical
convex cocompact manifolds, e.g. 
 for certain quasi-Fuchsian surface groups that are "nearly Fuchsian". 
Theorem \ref{int} in the special case of a doubly degenerate surface group shows that 
the Ratner-type phenomenon can be re-instated provided we deform the quasi-Fuchsian surface groups 
to a degenerate limit.
Pitching these two facts against each other seems to suggest an intriguing
possibility of a "phase-transition" in quasi-Fuchsian space, where surface groups that are "less degenerate"
do not exhibit Ratner-type phenomena, while surface groups that are "more degenerate"
do.

\subsection{Notation and Scheme} 
\begin{defn}\label{degdef}
We say that a hyperbolic 3-manifold $M$ is {\bf degenerate}, if all the ends of $M$ are degenerate.
The corresponding Kleinian group $\Ga$ will  be called a {\bf degenerate Kleinian group}.
\end{defn}

Note that a finitely generated Kleinian group $\Ga$ is degenerate if and only if its limit set 
$\Lambda (=\Lambda_\Ga)$ is all of $S^2$
or, equivalently, if the convex hull of $\Lambda$ is all of $\Hyp^3$.
Note also that if $\Ga$ is abstractly isomorphic to the fundamental group of a closed surface $S$ of genus greater than one,
then $\Ga$ is degenerate in the sense of Definition \ref{degdef} if and only if $\Ga$ is {\bf doubly} or
{\bf totally} degenerate in the sense of \cite{thurstonnotes}.  

We fix, once and for all,
 the following notation for this paper:
\begin{enumerate}
\item $G = \pslc$
\item $H = \pslr$
\item $A$ will denote the diagonal matrices in $G$.
\item $N$ will denote the upper triangular unipotent matrices.
\item $U$ will denote the {\it real} upper triangular unipotent matrices.
\item $V$ will denote the  upper triangular unipotent matrices with purely imaginary off-diagonal entry.
\item Throughout most of the paper,
$\Ga$ will denote a degenerate  Kleinian group, except in the last Section, where it will denote an arbitrary
finitely generated Kleinian group.
\item $M = \hyps / \Ga$ 
\item $K$ is a compact core of $M$.
\item $S^2$ will denote the boundary sphere of $\hyps$
\item $\SSS = \pslc / \pslr $ denotes the space of circles on $S^2$, equipped with the quotient topology.
\item $\UU = \pslc / U$ denotes the space of horocycles in $\hyps$, equipped with the quotient topology.
\item $X$ will be a minimal closed $H-$invariant subset of $M$.
\item $Y$ will be a minimal closed $U-$invariant subset of $M$.
\item $B^3 = \Hyp \cup S^2$ is the usual compactification of $\hyps$.
\end{enumerate}

\bigskip

 We are interested in studying closed $H-$invariant subsets of $M$. This is equivalent to 
studying closed $\Ga-$invariant subsets of $\SSS$ (see \cite{mmo} for more on this correspondence). We elaborate 
here on the topology on $\SSS$. Equip the collection of 
closed subsets of $S^2$ with the Hausdorff topology and let 
$C_c(S^2)$ denote the subspace consisting of subsets that have cardinality greater than one, i.e. we exclude singletons. Then $\SSS$ is contained in $C_c(S^2)$ and the topology coincides with the subspace topology.
Theorem \ref{int} then asserts  that
a minimal closed $\Ga-$invariant subset of $\SSS$ with the subspace topology inherited from 
$C_c(S^2)$  is either discrete or all of $\SSS$.

A recurring theme (cf. \cite{margulis-opp})
in the study of closed $H-$invariant sets $X$ is that it leads us naturally
 to the study of $U-$invariant
sets $Y$, where $U$ is the set of real unipotent matrices. $U-$orbits coincide with horocycles.
A word about the topology on $\UU$. Adjoining the basepoint of a horocycle to the horocycle we obtain a 
circle in $B^3$.  Equip the collection of 
closed subsets of $B^3$ with the Hausdorff topology and let 
$C_c(B^3)$ denote the subspace consisting of subsets that have cardinality greater than one, i.e. we exclude singletons. Then $\UU$ is contained in $C_c(B^3)$ and the topology coincides with the subspace topology.

We are now in a position to discuss the scheme of the paper. Section \ref{prel} recalls the two main ingredients
that feed into the paper:
\begin{enumerate}
\item The existence and
structure of Cannon-Thurston maps \cite{mahan-split, mahan-elct, mahan-red, mahan-kl} for arbitrary
finitely generated Kleinian groups.
\item Some general structure theorems from \cite{mmo}.
\end{enumerate}

Section \ref{hmin} establishes the existence of minimal closed $H-$invariant sets $X$. Our strategy
for the next three sections
is to construct within $X$, sufficiently complicated minimal closed $U-$invariant sets $Y$. Here `sufficiently complicated' simply means that $Y$ contains more than one horocycle. The construction
of $Y$ is indirect. Section \ref{horgeo} establishes a general correspondence between geodesic rays that
recur to a compact subset of $M$ and recurrent horocycles. This reduces the problem of constructing $Y$
to the construction of recurrent geodesic rays. Section \ref{lam} then  shows that (roughly speaking) the
intersection of $X$ with a suitably chosen compact core $K$ of $M$ furnishes an algebraic lamination
(traditionally the support of a geodesic current) $\CC$. Using the structure theory of Cannon-Thurston maps,
it then suffices to show that $\CC$ contains no ending lamination (cf.
 \cite{thurstonnotes, minsky-elc1, minsky-elc2}).  This, last, is accomplished by a geometric limit argument.

With the existence of closed $U-$minimal sets $Y$ in place, the proof of Theorem \ref{int}, carried out
in Section \ref{maint}, is largely similar to
 the corresponding argument in \cite{mmo}.
In Section \ref{fin}, we extend a Theorem of McMullen, Mohammadi and Oh \cite[Theorem B.1]{mmo}: we show
that in any infinite volume hyperbolic 3-manifold with finitely generated fundamental group, there can be only finitely many compact totally geodesic immersed surfaces.

 A final word about the difference in techniques between \cite{mmo} and the present paper. The former uses thick sets, the renormalized frame flow and mixing properties of the frame flow in an essential way,
as the focus there is on `reducing' the problem (in a sense) to the convex core, which in turn
has properties resembling a compact manifold. Th convex core
 is, in fact, a compact manifold with boundary. Our focus
in this paper, is topological-dynamic in flavor and a detailed  analysis of the $\Ga-$action on its
limit set (as explicated by the Cannon-Thurston map) is the main ingredient. By their very
nature, the examples dealt with in this paper permit no compact "reduction".
 In this sense, the examples
of hyperbolic 3-manifolds explored here are the first genuinely infinite volume examples
for which  a Ratner-type phenomenon is obtained.

\section{Preliminaries}\label{prel}

\subsection{Cannon-Thurston maps}

\begin{defn}   Let $W$ and $Z$ be hyperbolic metric spaces and
$i : Z \rightarrow W$ be an embedding. Suppose that a continuous extension $\hat{i}:
\widehat{Z} \to \widehat{W} $ of $i$ exists between their (Gromov) compactifications.
Then the boundary value of $\hat i$, namely
 $\partial i : \partial Z \to \partial W$ is called
a {\bf Cannon-Thurston map}.
\end{defn}

For us,
$Z$ will be a Cayley graph of the Kleinian group $\Ga$ with respect to a finite generating set
 and $W$ will be $\Hyp^3$, where we identify (the
vertex set of) $Z$ with an orbit of $\Ga$ in $\Hyp^3$. Equivalently
(as is often done in geometric group theory), we may choose a compact core $K$
of $M$ and $W, Z$ will be identified with $\til M$ and $\tK$ respectively.
Then the main Theorems of \cite{mahan-split, mahan-red, mahan-elct, mahan-kl} give us:

\begin{theorem} Let $M$ be a degenerate hyperbolic 3-manifold without parabolics and $K$
a compact core. Let $\Ga$ be the corresponding Kleinian group.
 A Cannon-Thurston map $\partial i$ exists for $i: \Ga \to \til{M}$
(or equivalently, $i: \tK \to \tM$). Further,
 $\partial i$ identifies $a, b \in \partial \Ga$ 
 iff $a, b$ are  end-points of a leaf of an ending lamination
$\LL_E$ or boundary points of an ideal polygon whose sides are leaves
of $\LL_E$, where $E$ is an end  of $M$. \label{ctstr}
\end{theorem}

In Section \ref{horgeo}, we shall be concerned with the conical limit set of a Kleinian group.
We furnish some definitions here and explicate the relationship with Cannon-Thurston maps.

\begin{defn} A point $z \in \Lambda (=\Lambda_\Gamma)$ is called a conical limit point if 
for any base-point $o$, there
exist 
\begin{enumerate}
\item a geodesic ray $\gamma$ landing at $z$
\item $R > 0$
\end{enumerate}
such that the neighborhood $N_R(\gamma)$ contains infinitely many distinct orbit points $g.o$, $g \in \Ga$.

The collection $\Lambda_c (=\Lambda_c(\Gamma))$  of conical limit points is called the conical limit set
of $\Ga$.
\end{defn}

\begin{remark} {\rm
The Cannon-Thurston map in Theorem \ref{ctstr}
furnishes a method of pulling back any subset of $S^2$ to $\partial \tK$. For any totally geodesic plane $\tP \subset \hyps$,
$(\partial i)^{-1} (\partial \tP) $ is a closed subset of $\partial \tK$, the Gromov boundary of $\tK$.
We denote $(\partial i)^{-1} (\partial \tP) $  by $\tP_K$ for short. Again, 
$(\partial i)^{-1} (\Lambda_c) $ is a  subset of $\partial \tK$, which we denote as $\Lambda_{c,K}$.
Section \ref{lam} can be viewed as an investigation into the properties of $\tP_K \cap \Lambda_{c,K}$; in particular we establish
that this intersection is non-empty.

A heuristic reason to believe that the intersection is non-empty for doubly degenerate surface Kleinian groups is the following: \\
A doubly degenerate $M$ has two ends, the $+\infty-$end and the $-\infty-$end say. Geodesics in $M$ that do not correspond to 
points in $\Lambda_c$ necessarily have to exit $M$ through either the $+\infty-$end or the $-\infty-$end.
It is therefore natural to expect that $\partial \tP$ cannot be built up entirely of these two types of points and anything
"in between" necessarily give conical limit points.} \label{pullback}
\end{remark}

\medskip

\begin{lemma} Let $P$ be a totally geodesic immersed plane in $M$ and $\tP$ be a lift to $\tM$.
Let $E$ be an incompressible end of $M$ and
 $S = \partial E$ be its boundary surface. Let $\partial{i} : \partial {\til{S}} \to \partial {\til{M}}$
denote the Cannon-Thurston map for $S \hookrightarrow M$. Then $\tP_S := (\partial{i})^{-1} (\partial \tP) $ is 
homeomorphic to a Cantor set contained in the circle $\partial \til{S}$. \label{cantor} \end{lemma}

\begin{proof} First, since $\partial i$ is continuous $\tP_S$ is closed. Further, since $S^1$ is perfect and $\partial i$ is finite-to-one,
$\tP_S$ is perfect. Also, for any interval $I \subset \partial \tS$, $\partial {i}$ identifies pairs of points corresponding 
precisely to end-points of 
leaves of $\LL_E$, the ending lamination corresponding to $E$ (by Theorem \ref{ctstr}. In particular, $I$ cannot map to a subarc of $\partial \tP$.
 In other words $\tP_S$ has empty interior.
\end{proof}

\subsection{Algebraic Laminations and Cannon-Thurston Maps}
\begin{defn}
An {\bf algebraic lamination} \cite{bfh-lam, mitra-endlam, chl08a, kl10} for a hyperbolic group $\Ga$ is a
$\Ga$-invariant,  closed subset $\LL \subseteq \partial^2
\Ga =(\partial \Ga \times \partial \Ga \setminus \Delta)/\sim$, where $(x,y)\sim(y,x)$ denotes the flip and $\Delta$ 
is the diagonal in $\partial \Ga \times \partial \Ga$. 

If $(p,q) \in \LL$, then any bi-infinite geodesic joining $p, q$ in $\Ga$ is called a leaf of $\LL$. \end{defn}

Cannon-Thurston maps, when they exist automatically define a natural algebraic lamination.
\begin{defn} \cite{mitra-pams} Suppose that a Cannon-Thurston map $\partial i$ exists for the pair $(\Ga, W)$.
We define the Cannon-Thurston lamination $\LL_{CT}$ as follows:\\
 $\LL_{CT} = 
 \{ (p,q) \in \partial^2 {\Gamma} | 
p \neq q$  and $\partial {i} (p) = \partial {i} (q) \}$.
\end{defn}

The relationship between algebraic laminations, Cannon-Thurston maps and quasiconvexity has been explored recently in \cite{mahan-rafi}.
We note the following facts that will be relevant for us.

\begin{lemma}\cite[Lemma 2.1]{mitra-pams} Let $\Ga$ be a  hyperbolic group acting on a hyperbolic metric space $W$ such that a Cannon-Thurston map exists
for the pair $(\Ga, W)$. Then some (any) $\Ga-$orbit is quasiconvex if and only if $\LL_{CT} =\emptyset$.\label{nonemptyiff}
\end{lemma}

\begin{lemma} Let $\Ga$ be a  hyperbolic group acting on a hyperbolic metric space $W$ such that
 a Cannon-Thurston map exists
for the map $i:\Ga \to W$ identifying $\Ga$ with its orbit in $W$.
 Let $\LL$ be an algebraic lamination  in $\partial^2 \Ga$. Then exactly 
one of the following holds:

\begin{enumerate}
\item Either there exists $C_0$ such that for every leaf $l$ of $\LL$, $i(l)$ is a $C_0-$quasigeodesic in $W$.
\item Or $\LL \cap \LL_{CT} \neq \emptyset$ and hence contains a minimal closed $\Ga-$invariant subset of $\LL_{CT}$. Further, $\Ga$ orbits are not quasiconvex.
\end{enumerate}
\label{current-lamn0}
\end{lemma}

\begin{proof} We identify $\Ga$ with a Cayley graph (with respect to a finite generating set).

Suppose that Alternative (1) fails. Then for all  $n \geq 0$, there exist leaves $l_n$ 
of $\LL$ and
 $a_n, b_n \in l_n$, $c_n \in [a_n, b_n]$
(the geodesic in $\Ga$ joining $a_n, b_n$) such that the geodesic $\alpha_n$ in $W$
joining $a_n, b_n$ lies outside  $ N_n (c_n) $, where $N_n(c_n)$
denotes the $n-$ball around $c_n$ in $W$. Translating by a group element $g_n$ (in
$\Ga$) if necessary, we may assume that that $c_n$ is the identity element of $\Ga$. 

Since $\alpha_n$ lies outside  $ N_n (c_n) $, the visual diameter of $\alpha_n$
tends to zero as $n \to \infty$. Hence, passing to a limit we obtain
\begin{enumerate}
\item a leaf $l_\infty$ of $\LL$ as a limit of the $l_n$'s (using the fact that $\LL$ is closed).
\item The Cannon-Thurston map for $i: \Ga \to W$ identifies the end-points at infinity
of $l_\infty$.
\end{enumerate}

Hence $l_\infty \subset \LL \cap \LL_{CT}$. In particular, $\LL_{CT}$ is non-empty and hence by Lemma \ref{nonemptyiff}, 
$\Ga$ orbits are not quasiconvex. Since $\LL_{CT}$ is also an algebraic lamination, $ \LL \cap \LL_{CT}$ contains the minimal closed
$\Ga-$invariant subset of $\partial^2 \Gamma$ containing $l_\infty$. Alternative (2) follows.
\end{proof}

\subsection{Zariski Dense Subgroups} The following two Theorems are specializations to our context of general facts from \cite{mmo}. The corresponding
Theorems in \cite{mmo} are established in the general context of Zariski dense subgroups of $\pslc$. This hypothesis is always satisfied for 
Zariski dense subgroups.

\begin{theorem}\label{openfull}\cite[Theorem 4.1, Corollary 4.2]{mmo}
Let $\Ga$ be a degenerate Kleinian group and let $\Lambda (=S^2)$ be its limit set.
Let $\DD \subset \SSS$ be a collection of circles such that  $\cup_{D \in \DD} D$ contains a
nonempty open subset of $\Lambda$. Then there exists a $D \in \DD$ such that $\bbar{\Gamma D} = \SSS$.\end{theorem}

\begin{theorem}\label{zdcor}\cite[Theorem 5.1]{mmo} Let $\Ga$ be a degenerate Kleinian group.
Let $C \subset  S^2$ be a circle, and suppose $D \in \overline{\Gamma C} \setminus  {\Gamma C}$. 
Let $Stab(D)$ denote the stabilizer of $D$ in $\Ga$. If 
$Stab(D)$ is a non-elementary group, then $\overline{\Gamma C} = \SSS.$
\end{theorem}

The next statement is valid for arbitrary hyperbolic 3-manifolds.

\begin{theorem} \label{bddplane} \cite[Corollary B.2]{mmo}
 Let $N$ be an arbitrary hyperbolic 3-manifold. Let $K \subset N$
 be a compact subset.
Then either the
 set of compact geodesic surfaces in $N$ contained in $K$ is finite; or the union of
compact geodesic surfaces is dense in $N$, and $N$ is compact.
 \end{theorem}

We shall also need the following general  Lie groups fact proved in \cite{mmo}:

\begin{theorem} \cite[Theorem 8.2]{mmo} Let $g_n \to id$ in $G\setminus AN$ and $G_0$
be a neighborhood of the identity. Then there exist $u_n,u_n'
 \to \infty $ in $U$
such that after passing to a subsequence, we have
$u_ng_nu_n' \to g \in G_0 \cap (AV − {id})$.\label{av} \end{theorem}

\section{H-minimal sets}  \label{hmin}

\begin{lemma} Let $M$ be a degenerate manifold and $K$ a compact core of $M$.
Then any totally geodesic immersed plane in $M$ intersects $K$. \label{intersect} \end{lemma}

\begin{proof}  Let $f : \hypp \to M$ be a totally geodesic  immersion
and $\til{f}: \hypp \to \hyps$ be a lift to the universal cover.

Let $K$ be the compact core and $i: K \to M$ denote inclusion. Then $i$ lifts to
$\til{i} : \til{K} \to \til{M}$. 

If $f(P)$ misses $i(K)$, then there exists a lift
$\til{f} (\hypp)$ missing $\til{i} ( \til{K})$, and so the limit set of $\til{i} ( \til{K})$
lies  entirely to one side of the boundary circle  $C = \partial (\til{f} (\hypp))$. But then the limit set
of $\til{K}$ is not all of $S^2$, contradicting the fact that $K$ is the compact core of a degenerate
manifold $M$.
\end{proof}

Since a (not necessarily totally geodesic plane) lying in a bounded neighborhood of a totally geodesic plane
necessarily has the same limit set, we have the following immediate Corollary:

\begin{cor} Let $M$ be a degenerate manifold and $K$ a compact core of $M$.
Let $P$ be a totally geodesic immersed plane in $M$, $\tP$ be a lift to $\tM$, $\tP_1$ be a plane lying in a bounded
neighborhood of $\tP$ and $P_1$ be the image of $\tP_1$ in $M$ under the covering projection.
Then $P_1$ intersects $K$. \label{intersectcor} \end{cor}

\begin{prop} Let $M$ be a degenerate manifold.
Let $X \subset M$ be a closed subset that is $H-$invariant. Equivalently, $X$ is a closed subset of $M$ that is
 a union of immersed
totally geodesic planes in $M$. Then there exists a minimal closed $H-$invariant subset $X_0 \subset X$. 
\label{h-exist} \end{prop}

\begin{proof} The proof is standard (cf. \cite{margulis-opp, mmo}) given Lemma \ref{intersect}. Let $K$ denote a compact core of $M$.
Any  closed $H-$invariant $X_\alpha$ intersects $K$ by Lemma \ref{intersect}.
We consider the collection of closed $H-$invariant subsets $X_\alpha \subset X$ partially ordered as follows: $X_\alpha < X_\beta$
if $X_\alpha \cap K \subset X_\beta \cap K$.  
Hence, for any totally ordered collection  $\{ X_\alpha \}$, $\cap_\alpha X_\alpha \cap K$ is non-empty, as $K$ is compact. 
This forces $\cap_\alpha X_\alpha$ to be non-empty. The existence of a  minimal closed invariant subset $X_0 \subset X$
now follows by Zorn's Lemma.
\end{proof}

A minimal closed $H-$invariant subset will be called {\bf $H-$minimal}. We have thus proved the existence 
of $H-$minimal sets. We assume henceforth (using Proposition \ref{h-exist}) that $X=X_0$
is $H-$minimal.
 Recall that $U$  denotes the {\it real} upper triangular unipotent matrices. The next few sections
of the paper are devoted to finding sufficiently complicated $U$-minimal sets inside the $H-$minimal set $X$.
This will involve a number of ingredients.

\section{Horocycles and geodesics}\label{horgeo}
 In this section, we relate properly embedded geodesics in $M$ with 
properly embedded horocycles in $M$. We note here that
for the purposes of this section, we assume only that $M$ is degenerate without parabolics; no hypothesis on incompressibility
of its compact core $K$ are imposed.
Properly embedded geodesics in $M$ are said to be
{\bf exiting}.
A geodesic ray in $M$ is exiting if and only if any lift to the universal cover converges
to a point $p \in \partial \til{M}$ not belonging to the conical limit set (cf. \cite[Section 3]{mahan-lecuire}).
It follows that if  $\gamma$ is a geodesic ray in $M$ such that some (and hence every) lift $\gamma_1$
of $\gamma$ to $\til M$ lands on a conical limit point, then there exist
\begin{enumerate}
\item a compact subset $Q \subset M$
\item a sequence of times $t_n \to \infty$
such that $\gamma (t_n) \in Q$.
\end{enumerate}

 In fact non-exiting geodesics are {\em characterized} by the fact that they recur
to a compact subset of $M$ infinitely often.  In other words, geodesics in $M$ that recur infinitely often
to a compact subset of $M$ correspond precisely to geodesics in $\til M$ that land on the conical limit set.
We refer the reader to \cite[Section 3]{mahan-lecuire} for a detailed discussion. In Proposition
\ref{geod-horoc} we shall show that proper embeddedness of a horocycle forces the corresponding
geodesic ray to be exiting. In subsequent sections of this paper,
 we are going to be interested in horocycles that are  {\bf not}  properly embedded.
Proposition \ref{geod-horoc} then tells us that it suffices to look at geodesics that are not exiting.

Recall that $K$ is the compact core of $M$. Let $X$ be $H-$minimal. 
Let $P \subset X$ be one of the totally geodesic planes comprising $X$.
By minimality the closure $\bbar{P}=X$.
Let $\sigma (\subset P) $ be a (bi-infinite) horocycle contained in $P$ passing through the compact core $K$.
Choose $o \in K \cap \sigma$ such that $o$ does not lie in the self-intersection locus of $P$. It follows that
after lifting to the universal
cover $\til M$ 
and fixing a lift $o_1$ of $o$, there is a unique lift of $P$ through $o_1$. Call this $P_1$. Let $\sigma_1$
be the lift of $\sigma$ through $o_1$. Then $\sigma_1 \subset P_1$. Let $\tau_1$ be the geodesic ray  in $P_1$
starting at $o_1$ and asymptotic  to $\sigma (+\infty) = \sigma (-\infty) (\in \partial \til{M})$.
Project $\tau_1$ back to $M$ to obtain the  geodesic ray
 $\tau $   through $o$, contained in the (immersed) horodisk in $P$ bounded by $\sigma$. We say that
$\tau$ is the unique geodesic ray through $o$ in $P$ that is {\bf asymptotic} to $\sigma$.

\begin{prop} Let $\sigma, P, o$ be as above and let $\tau$ be
 the unique geodesic ray through $o$ in $P$  asymptotic to $\sigma$. If $\sigma$ is properly embedded in $M$,
then so is $\tau$. \label{geod-horoc} \end{prop}

\begin{proof}
As in the discussion before the Proposition, let $o_1, \sigma_1, \tau_1, P_1$ be lifts of $o, \sigma, \tau$ and 
$P$
respectively. Let $D_1$ be the horodisk with $\sigma_1$ as its boundary and $D$ denote its projection to $M$.

If $a_n \to \infty$ and $b_n \to -\infty$ along $\sigma$, the hypothesis guarantees that $a_n, b_n$ exit 
the same degenerate end $E$ of $M$. Let $(a_n, b_n)$ denote the subarc of $\sigma$ joining $a_n, b_n$
and let $[a_n, b_n]$ be the geodesic in $D$ joining $a_n, b_n$. 

We first show that $[a_n, b_n]$ exits $E$.
If not, then there is a compact $Q \subset M$ such that $[a_n, b_n] \cap Q \neq \emptyset$ for all $n$. Passing
to a subsequential limit we obtain a geodesic $\gamma$ passing through $Q$ and hence (after lifting to $\til{M}$)
$\gamma (\infty)
\neq \gamma (-\infty)$, where $\gamma (\infty) = lim_n a_n$ and $\gamma (-\infty) = lim_n b_n$ in $\til M$. 
On the other hand any subsequential limit of $(a_n, b_n)$ is $\sigma$ (since 
 $a_n \to \infty$ and $b_n \to -\infty$ along $\sigma$) and hence (after lifting to $\til{M}$)
the limits of $a_n$ and $b_n$ must coincide. This contradiction proves that  $[a_n, b_n]$ exits $E$
for any $a_n \to \infty$ and $b_n \to -\infty$ along $\sigma$.
Now, choose $a_n = n$ and $b_n = -n$ and let $t_n = \tau \cap [a_n, b_n]$. Since  $[a_n, b_n]$ exits $E$,
it follows that $t_n$ exits $E$. It follows that $\tau$ exits $E$ (since the distance between $t_n$ and $t_{n-1}$
is uniformly bounded and hence the geodesic joining them in $\til M$ is uniformly bounded in length).
\end{proof}

\section{Algebraic Laminations}\label{lam} 
Our aim (till Proposition \ref{u-exist} below) is to establish the existence of a sufficiently complicated
$U-$minimal set $Y$ contained in the $H-$minimal set $X$
constructed in Proposition \ref{h-exist}. Here "sufficiently complicated" simply means that $Y$ does not consist
of a single horocycle. 

Our approach in constructing such a $Y$ is indirect.
We shall be interested in horocycles that are not properly embedded in $M$.
Proposition \ref{geod-horoc} allows us to turn our attention instead at geodesics that are not exiting. A natural
class of non-exiting geodesics are given by those that lie in a compact subset of $M$. 
The purpose of this section is to construct such a class of geodesics. 
Let $K$
be a compact core of $M$ as usual. Roughly speaking, the intersection $X \cap K$ furnishes for us
such a family of geodesics.

 Let $\til X$, $\til K$ and $\tP$ be pre-images of $X, K$ and $P$
respectively in $\til M$. If some $P (\subset X)$ is contained in a compact subset $Q$ of $M$,
then by minimality of $X$, so is $X$. It follows that all horocycles contained in planes comprising $X$ are contained in $Q$, 
and hence none are  properly embedded in $M$. In this situation, we define $Y_0$ to be the $U-$invariant subset of $G/U$ given by {\bf all} 
horocycles contained in planes comprising $X$. Proposition \ref{u-exist} will establish the existence of a $U-$minimal set $Y$ contained in $Y_0$ 
in this situation without much further work (the proof is a reprise of that of Proposition \ref{h-exist}).

In light of this we work under the following hypothesis for the purposes of this section:

\begin{hypo} \label{hypoth} The $H-$minimal set $X$ furnished by Proposition \ref{h-exist} is not contained in any compact subset
of $M$.\end{hypo}

\begin{remark} Theorem \ref{bddplane} by McMullen, Mohammadi and Oh (or an extension of
\cite{ratner-top}) guarantees that Hypothesis \ref{hypoth}
holds for degenerate $M$: \\Either $M$ is compact or $X$ is a closed immersed surface. 

However,
we do not need to apply Theorem \ref{bddplane} in order to prove Theorem \ref{main} below; hence the status
of a Hypothesis. \end{remark}

\noindent {\bf Notation:} For $K$ a compact manifold with boundary, we shall denote the boundary by
$bdy(K)$. Thus, for $K$ a compact core of $M$,  $\btK$ will denote the boundary of $\til K$
thought of as a manifold with boundary. We shall reserve the notation $\partial \til{K}$ to denote
the {\bf Gromov-}boundary of $\til K$, when the latter is hyperbolic.

\medskip

As usual, let $K$ be a compact core of $M$. Consider a lift $\tP$ of $P$ to $\tM$. 
For any component $K_0$ of $\tP \cap \tK$, we define the {\bf depth} $inj(K_0)$ of $K_0$ to be $sup_{x \in K_0} d(x,  \btK)$. Thus, the depth
of $K_0$ equals the radius of the largest totally geodesic hyperbolic disk that can be embedded in $K_0$. This makes sense as $K_0$ is a closed subset of the totally
geodesic plane $\tP$ in $\tM$.

\begin{lemma} \label{ndc} {\bf No deep components:} Fix an $H-$minimal $X$ in $M$ and assume that Hypothesis \ref{hypoth} holds.
Then for any compact core $K$ of $M$, there exists $R_0$ such that for any $P$ in $X$ and any component $K_0$ of $\tP \cap \tK$, 
 $inj(K_0) \leq R_0$. \end{lemma}

\begin{proof} Suppose not. Then for all $n \in \natls$, there exists a plane $P_n$, a lift $\tP_n$, a component $K_0$ of $\tP_n \cap \tK$,
and $z_n \in K_0$ such that a totally geodesic disk of radius $n$ about $z_n$ is contained in $K_0$. Translating by an element of $\Ga (=\pi_1(M)
= \pi_1(K))$, we may assume that $z_n$ lies inside a fundamental domain  for the $\Ga-$action on $\tK$. Passing to a limit,  as $n \to \infty$,
we obtain a totally geodesic (infinite) plane $P_\infty$ contained inside $K$. Since $X$ is $H-$minimal, it equals the closure of 
$P_\infty$ and hence $X \subset K$, contradicting Hypothesis \ref{hypoth}.
\end{proof}

\subsection{Intersections of $\tX$ with $\btK$} In this subsection, we give a topological argument to show that `spurious' intersections
of $\tX$ with $\btK$ can be removed by standard topological surgeries.\\

\noindent {\bf Removing Inessential Loops:} Consider a lift $\tP$ of $P$ to $\tM$
and assume (after perturbing $K$ slightly in $M$) that $\tP$ is transverse to $\btK$. A compact connected component $\sigma$
(necessarily homeomorphic to $S^1$)
of $\tP \cap \btK$ is called an {\bf inessential loop} if $\sigma$ is homotopically trivial in $\btK$. Equivalently,
$\sigma$ bounds a disk in the boundary $\btK$.  We may replace 
every bounded component $K_0$ of $\tP \cap  \tK$ by a subsurface $K_1$ of $\tP \cap \btK$ having the same boundary circles. 
Since there are no deep components by Lemma \ref{ndc} (under Hypothesis \ref{hypoth}), $inj(K_0) \leq R_0$ for all $K_0$.
Since $\btK$ is uniformly properly embedded in $\tM$, it follows immediately that the injectivity radius of $K_1$ is bounded 
in terms of $R_0$. Hence replacing each such $K_0$ by the corresponding $K_1$, we obtain a surface $\tP_1$ lying in a bounded neighborhood of $\tP$. 
Also they have the same boundary circle $\partial \tP_1 = \partial \tP$. Replacing $\tP$ by $\tP_1$ (not necessarily totally geodesic), and taking the
minimal closed $\Ga-$invariant set generated by $\tP_1$, we obtain a new minimal set $X_1$ in $M$ having no inessential loops in $\tK \cap \til{X_1}$. 

Note that though each $ \til{X_1}$ is no longer $H-$invariant, the collection $\{ \partial \tP_1 : P_1 \in X_1 \}$ {\it does} agree with $\bbar{\Ga \partial {\tP}}
\subset \SSS$.  Since $ \til{X_1}$ will serve only an auxiliary purpose in what follows, this is not going to cause a problem in what follows. \\

\noindent {\bf Removing Asymptotically Inessential Loops:} Let $P, \tP, K, \tK$ be as above; in particular $\tP$ is transverse to $\btK$.
A bi-infinite path $\sigma$ in $\tP \cap \btK$ is called a {\bf $C_0-$asymptotically inessential loop} if it is not properly embedded, or equivalently
in our situation, there exist
$a_n \to \infty, b_n \to - \infty$ along $\sigma$ such that $d(a_n, b_n)$ is bounded by $C_0$. 

Such a bi-infinite path $\sigma$ gives "almost closed loops" in the following sense. There exists  arcs 
$\tau_n$ (resp. $\theta_n$) of  length bounded in terms of $C_0$
in $\btK$ (resp. $\tP$) such that

\begin{enumerate}
\item $\tau_n$ (resp. $\theta_n$) along
with arbitrarily long subarcs $\sigma_n$ joining $a_n, b_n$  bound contractible loops $\alpha_n$ (resp. $\beta_n$) in $\btK$ (resp. $\tP$). The loops $\alpha_n$ (resp. $\beta_n$) bound disks $\Delta_{1n}$ (resp. $\Delta_{2n}$)
\item the loops  $\tau_n \cup \theta_n$ 
bound disks $\Delta_{3n}$ lying in a bounded neighborhood of $\btK$.
\end{enumerate}

Again by homotoping $\tP$ by a bounded amount $D_0$ (depending only on $C_0$), one can get rid of $C_0-$asymptotically inessential loops. This is done iteratively over $n$: replace $\Delta_{2n}$ by $\Delta_{1n}\cup \Delta_{3n}$
and let $n \to \infty$.\\

\noindent {\bf Two Alternatives:}
We are now in a position to describe  intersections of $\tX$ with $\btK$. We shall note below
that we can reduce $\tX$ with $\btK$ to essentially two kinds of intersections
that survive removal of inessential loops and asymptotically inessential loops:

\begin{enumerate}
\item[Alternative A:] A (closed $\pi_1(M)-$invariant) collection $\CC$ of 
 bi-infinite paths in $ \tX \cap \btK$ properly embedded in $\btK$. 
\item[Alternative B:] A (closed $\pi_1(M)-$invariant)
 collection $\DD$ of  compressing disks in $ \tX \cap \btK$  where each $D \in \DD$ has diameter bounded by some $d_0$.
\end{enumerate}

We describe now how the two above alternatives are obtained. First, by homotoping $K$ and $X$, we can remove
inessential loops. Next, for any fixed $C_0$, we can get rid of $C_0-$asymptotically inessential loops
by a similar homotopy. Two cases arise now. If all asymptotically inessential loops are 
$C_0-$asymptotically inessential for some $C_0 > 0$, then we can remove them all. Else, 
$\tX\cap\btK$ contains (Gromov-Hausdorff) 
limits of $n-$asymptotically inessential loops, as $n$ tends to infinity. Such limits are necessarily 
 bi-infinite paths in $ \tX \cap \btK$ properly embedded in $\btK$. In short, either we can remove
all asymptotically inessential loops or Alternative A holds.

Suppose therefore that  Alternative A fails to hold and also (by homotopy) inessential loops as well as 
asymptotically inessential loops have been removed. We shall show that in this situation, 
Alternative 2 holds. Since $\CC$ is empty, all intersections come from simple closed
curves in $\btK$ that are compressible in $M$, and hence in $K$. The innermost curves necessarily 
correspond to compressing disks in $\tX\cap\btK$. If there exists a sequence $D_i$ in
this collection, such that $dia(D_i) \to \infty$, then, by Lemma \ref{ndc}, they have uniformly bounded depth
and  their boundary circles necessarily limit to
paths in $\btK$ satisfying Alternative A, contradicting our assumption. Hence 
Alternative 2 holds. We summarize this discussion as follows:

\begin{prop} \label{alter} Let $M$ be a degenerate hyperbolic 3-manifold with $K$ a compact core.
Let $X$ be an $H-$minimal set in $M$ not contained in a compact subset. Then $ \tX \cap \btK$
satisfies at least one of the following alternatives:
\begin{enumerate}
\item[Alternative A:] $ \tX \cap \btK$ contains a (closed $\pi_1(M)-$invariant) collection $\CC$ of 
 bi-infinite paths in $ \tX \cap \btK$ properly embedded in $\btK$. 
\item[Alternative B:] $ \tX \cap \btK$ contains a (closed $\pi_1(M)-$invariant)
 collection $\DD$ of  compressing disks in $ \tX \cap \btK$  where each $D \in \DD$ has diameter bounded by some $d_0$.
\end{enumerate}
\end{prop}
As an immediate consequence, we have:

\begin{lemma} Suppose that $M$ has an incompressible core, i.e. $\bK$ is incompressible in $M$. Then,
$ \tX \cap \btK$
contains a (closed $\pi_1(M)-$invariant) collection $\CC$ of 
 bi-infinite paths in $ \tX \cap \btK$ properly embedded in $\btK$. \label{nonempty} \end{lemma}

We now proceed to deal with Alternative A.
Each bi-infinite path $l$ in $\CC$ lifted to $\btK$
will be called a {\bf leaf} of $\CC$.

\begin{lemma}  There exists $C$ such that for any leaf $l$  of $\CC$ lifted to $\btK$ and any $a, b \in l$,
$l \subset N_C ([a,b])$, where $N_C ([a,b])$ denotes the $C-$neighborhood of the 
 geodesic $[a,b]$ joining $a, b$ in $\btK$. 
\label{ccqgeod} \end{lemma}

\begin{proof} 
Suppose  not. Then for every positive integer $n$, there exist leaves $l_n$ and
 $a_n, b_n \in l_n$, $c_n \in [a_n, b_n]$
(the geodesic in $\tK$ joining $a_n, b_n$) such that $l \cap N_n (c_n) = \emptyset$, where $N_n(c_n)$
denotes the $n-$ball around $c_n$ in $\tK$. 
Since $l_n$ is a (necessarily connected) path in $\btK$, we may assume that there exists a component
$\tS$ of $\btK$, such that $[a_n, b_n]$ tracks a quasigeodesic in $\tS$. Without loss of generality
therefore, we assume (after passing to a subsequence if necessary)
that $[a_n, b_n]$ is contained in $\tS$ for all $n$. 

Translating by a group element $g_n$ (in
$\pi_1(K)$) if necessary, we may assume that that $c_n$ lies in a fixed fundamental domain in $\tK$. Passing 
to a subsequence we have a sequence of paths $\alpha_n$ in $\tS$
 joining $a_n, b_n$ and lying outside an
$(n-D_0)-$ball about a fixed base point, where $D_0$ is the diameter of a fundamental domain in $\tK$.
Let $P_n \subset \tM$ be the sequence of totally geodesic planes containing $\alpha_n$. Passing to a subsequence
we may assume that the sequence $\{ P_n \cup \partial (P_n)\}$ 
converges in the Hausdorff topology on compact subsets of
$\til{M} \cup \partial \til{M} = B^3$. Since each $\{ P_n \cup \partial (P_n)\}$ is a round 
disk (say in the projective model), the limit $P_\infty \cup \partial P_\infty$ is necessarily
 a round disk passing through a fixed fundamental domain in $\tK$.

In particular, the preimage under the Cannon-Thurston map $\partial i$
of  $\partial P_\infty$ is  a Cantor set in the boundary $\partial \tS$ by Lemma \ref{cantor}.
On the other hand $[a_n, b_n]$ converges to a bi-infinite geodesic $(a_\infty, b_\infty)$
and hence $\alpha_n$ converges to an arc in $\partial \tS$ joining $a_\infty, b_\infty$
in $\partial \til{S}$.
This contradiction yields the Lemma.
\end{proof} 

As an immediate Corollary we have

\begin{cor}  There exists $C$ such that  any leaf $l$  of $\CC$ in $\tK$ is a $C-$quasigeodesic
in $\tK$. 
\label{ccqgeodcor} \end{cor}

\begin{remark} \label{postcor} {\rm
We shall say that a leaf $l$ is {\bf carried by} a component
$\tS (\subset \btK)$ if it has a quasigeodesic representative contained in $\tS$. In this situation
we shall identify $l$ with such a quasigeodesic representative.
By Corollary
\ref{ccqgeodcor}, we may assume without loss of generality that $\CC$ is an algebraic lamination, i.e. a closed
$\pi_1(K)-$invariant collection of bi-infinite geodesics in $\btK$, or equivalently, a 
closed
$\pi_1(K)-$invariant subset of $\partial^2 \Ga$. Further, the proof of Lemma \ref{ccqgeod} shows that
there exists a component $S$ of $\bK$, such that its preimage in $\tK$ carries a closed $\pi_1(K)-$invariant
subset of $\CC$ (when Alternative A holds).}
\end{remark}

\subsection{Quasigeodesics in $\til{M}$} We have already shown that 
leaves of $\CC$ are uniform quasigeodesics in $\til K$. 
In this subsection we show that, moreover, 
leaves of $\CC$ are uniform quasigeodesics in $\til M$ (when Alternative A holds). 

\begin{lemma} Let $M$ be a degenerate hyperbolic manifold and let $S$ be a boundary component
of a compact core $K$ of $M$. Let $\CC_0$ be an algebraic lamination in $\partial^2 \Gamma$, whose leaves
are carried by $\tS$. Then exactly 
one of the following holds:

\begin{enumerate}
\item Either there exists $C_0$ such that every leaf of $\CC_0$ is a $C_0-$quasigeodesic in $\tM$.
\item Or there exists an end $E$ such that $\LL_E \subset \CC_0$, where $\LL_E$ is the ending lamination
for an end $E$ with $\partial E= S$.
\end{enumerate}
\label{current-lamn}
\end{lemma}

\begin{proof} 
Suppose that alternative (1) does not hold. Then, by Lemma \ref{current-lamn0}, $\CC_0 \cap \LL_{CT} \neq \emptyset$,
where $\LL_{CT}$ denotes the Cannon-Thurston lamination for the action of $\Ga$ on $\hyps$.

By Theorem \ref{ctstr}, there exists a boundary component $S$ of the compact core $K$ such that
 $l_\infty$ is either a leaf of an ending lamination $\LL_E$
for some end $E$ with boundary $S$ 
or a diagonal of a complementary
ideal polygon. Since the smallest closed $\pi_1(S)-$invariant subset of $(S^1 \times S^1 \setminus \Delta)/\sim$
containing such an $l_\infty$ is all of $\LL_E$, it follows that alternative (2) holds.
\end{proof}

\begin{lemma} Let $\CC_0$ be as in Lemma \ref{nonempty}. Then there exists $C_0$ such that every leaf of $\CC_0$ is a $C-$quasigeodesic in $\tM$. \label{qgeod} \end{lemma}

\begin{proof}
Suppose not. Then, by Lemma \ref{current-lamn}, there
exists an end $E$ such that $\LL_E \subset \CC_0$, where $\LL_E$ is the ending lamination
for an end $E$ with $\partial E= S$.

Let $\Theta$ be a complementary ideal polygon of $\LL_E$ in $\tS$. Let $\gamma_i: i =1,2, \cdots, k$ be the
 (infinite) sides of $\Theta$. Note that by Theorem \ref{ctstr}, all the points 
$\partial{i}(\gamma_i (\pm \infty))$ are the same. Let $z$ denote this point on $\partial \til{M}$.
Let $P_i$ be the totally geodesic plane on which $\gamma_i$ lies. There exist $x_i \in \gamma_i$
such that the diameter of the set $\{ x_1, \cdots , x_k \}$ is bounded by $k \delta$ (where $\delta$
is the hyperbolicity constant of $\tS$). Hence the geodesic rays $[x_i, z)$ lying on $P_i$ and asymptotic to $z$
all lie in $k \delta$ neighborhoods of each other. Next for any two successive sides
$\gamma_i, \gamma_{i+1}$ there are infinite rays contained in each (say the forward directed ray $\gamma_{i,+}$ in $\gamma_i$
and the backward directed ray $\gamma_{i+1,-}$ in $\gamma_{i+1}$) which lie at bounded distance from each other. Hence 
the convex hull $H_{i,+}$ of $[x_i, z) \cup \gamma_{i,+}$ (which lies in $P_i$) and 
the convex hull $H_{i+1,-}$ of $[x_{i+1}, z) \cup \gamma_{i+1,-}$
 (which lies in $P_{i+1}$) also lie at bounded distance
from each other. 

Translate
all the $H_{i,+}, H_{i,-}$'s by a sequence of elements $g_n$
 of $G (= PSL_2 (\C))$ that translate along $[x_1, z)$ by pulling
 a point $p_n \in [x_1, z)$ with $d(p_n, x_1) = n$ back to $x_1$.
Then, since $H_{i,+} \cup  H_{i,-} \subset P_i$,
 we obtain in the limit a family of totally geodesic planes $P_{i, \infty}$, one for each $i$.
Hence we have, in the limit, $k$ round circles $\partial P_{i, \infty}$ on $\partial \til{M}$. (Note here that the elements
$g_n$ do not lie in $\Ga$, but rather in $G=\pslc$.)


On the other hand, by the structure of Cannon-Thurston maps (Theorem \ref{ctstr}), the limit 
of $\cup_i (H_{i,+} \cup H_{i,-})$ has as its boundary  the one point compactification of
the k-pronged singularity corresponding to $\Theta$, or equivalently a space $\bbar \Theta$ homeomorphic to
the suspension of $k$ points. Since $\bbar \Theta$ cannot be homeomorphic to the union of $k$ 
round circles, we have a contradiction. 
\end{proof}

\begin{remark} \label{nonmin} We note here that the proof of Lemma \ref{qgeod} above {\bf does not} use 
 minimality of $X$. We first give the analogous statement
when $X$ is an $H-$invariant $X$ generated by a single plane $P$, 
i.e. $\tX = \bbar{\Ga \tP}$. (This is adequate for the removal of inessential and asymptotically inessential loops.)
We  use the $X$ thus obtained to define $\CC$ as in Lemma \ref{nonempty}.
The proof of Lemma \ref{qgeod} now goes through to establish that  leaves of $\CC$ are uniform
quasigeodesics in $\tM$. 

Reverse engineering the argument now, suppose that a leaf $l$ of the ending lamination $\LL_E$ is in $\CC$
for some $H-$invariant $X_0$ properly contained in $M$. Then 
$l \subset \tP \cap \btK$ for some $P \in X_0$. Let $X$ be the $H-$invariant set generated by  $P$. Applying the argument in the previous
paragraph shows that leaves of $\CC$ are uniform
quasigeodesics in $\tM$. In particular, no leaf of $\LL_E$ is in $\CC$ - a contradiction. We thus conclude that
the conclusions of Lemma \ref{qgeod} remain valid for any closed $H-$invariant $X$ properly contained in $M$.
 \end{remark}

\section{The Main Theorem}\label{maint}
We are now in a position to define $U-$minimal sets contained in the $H-$minimal set $X$.
Let $\CC_0$ be the algebraic lamination furnished by Lemma \ref{ccqgeod}. Lemma \ref{qgeod} then guarantees
that leaves of $\CC_0$ are uniform quasigeodesics in $\tM$.
Remark \ref{postcor} then guarantees that a sublamination $\CC$ of $\CC_0$ is carried by
a boundary component $S \subset \bK$.  We replace each  leaf $l'$ of $\CC$ by the
unique bi-infinite
geodesic $l$ in $\tM$ that tracks it.
Let $Y_0$ be the set of all (bi-infinite) horocycles $\sigma$ given by the following:\\
For each triple $\{(P, l, p): P \in X, l \in \CC, p \in l\}$ there are precisely two horocycles $\sigma_\pm 
(P,l,p)$ contained in $\tP$
passing through $p$ and converging at infinity  to $l(\pm \infty)$. 
The union of all such horocycles is denoted by $Y$. 

By Proposition \ref{geod-horoc}, none of these
horocycles are properly embedded in $M$.
Since the topology on horocycles consists of the topology of circles in $B^3$ tangential to $S^2$ minus singletons,
$Y_0$ is closed (since $\CC$ is closed as is $X$) and equivariant under $\pi_1(S)$.

The same argument as in Proposition \ref{h-exist} now gives: 
\begin{prop}  There exists a minimal closed $U-$invariant subset $Y \subset Y_0$. Further, for all $y \in Y$,
$Y=\bbar{Uy}$.
\label{u-exist} \end{prop}

\begin{proof} Each horocycle $\sigma$
meets a bi-infinite geodesic $l (\in \CC)$ at right angles in some $P \in X$. Here $l$ lies at a 
uniformly bounded distance from $\tS$. Hence there exists a compact set $Q \subset M$ such that the image
of $\sigma$ in $M$ passes through $Q$ infinitely often. We partially order closed $U-$invariant subsets $U_\alpha$ of
$Y_0$ as follows:
$U_\alpha < U_\beta$ if $U_\alpha \cap Q \subset U_\beta \cap Q$. Since $Q$ is compact, any totally ordered
chain has a lower bound (by taking intersections). Hence, by Zorn's Lemma, there exists a minimal closed 
$U-$invariant set.
\end{proof}

\subsection{Reduction to two cases}

We now state the main theorem of the paper. The rest of the Section is devoted to its proof.
\begin{theorem} \label{main} Let $\Gamma$ be a degenerate Kleinian group without parabolics and $M = \Hyp^3/\Gamma$ be the 
associated degenerate hyperbolic 3-manifold. Suppose further that $M$ has an incompressible core. Let $X$ be a 
minimal closed 
$H-$invariant subset, where $H=\pslr$. Then $X$ is either an immersed totally geodesic
surface  or all of $M$.
\end{theorem}

After obtaining closed minimal $U-$invariant sets $Y$ from Proposition \ref{u-exist},
 we shall follow the overall plan of \cite{margulis-opp, mmo}. 
We thus obtain a sequence $g_n \in (G\setminus H)$
 such that $g_n \to 1$ and $g_nY = Y$. From this, using Theorem \ref{av},
 we extract a sequence  $v_n \in AV$ such that $v_n \to 1$ and $v_nY = Y$.

First since $Y$ is closed minimal and $U-$invariant, it follows that for every $y \in Y$, $\bbar{Uy} = Y$. Further,
Proposition \ref{geod-horoc} guarantees that $Y \neq Uy$, i.e. the $U-$minimal set $Y$ {\bf does not}
consist of  single orbit.
We now want to
find a sequence of small elements $g_n \in (G\setminus H)$
 such that $g_nY = Y$. By minimality, it suffices to find $g_n$ such that 
$g_nY \cap Y \neq \emptyset$.

Fix a horocycle $\sigma$ in $Y$. Then by minimality of $Y$, $\bbar{\sigma} = Y$. 
Let $P$ be the immersed totally geodesic plane containing the horocycle $\sigma$ in $Y$.
 Then there exists a sequence $\{ z_n \}$ satisfying the following.
\begin{enumerate}
\item $z_n \in Uy$.
\item $\{ z_n \}$  is Cauchy in $M$. 
\item There is a sequence of geodesics $[z_n,z_m]$ joining $z_n, z_m$ in $M$
 such that the lengths $l([z_n,z_m])$ tend to zero. 
\end{enumerate}
By choosing the length of the horocycle segment between $z_n, z_m$ large enough, we can assume that
 the unique geodesic $(z_n,z_m)$ lying on $P$, joining $z_n$ to $z_m$,  and path-homotopic to
the corresponding horocycle segment (contained in $Uy$)
is almost perpendicular
to $Uy$ at its end-points.  Hence the two ends of $(z_n,z_m)$ 
 are nearly parallel. Joining them by $[z_n, z_m]$ gives a closed loop $\gamma_{mn}$ arbitrarily close to a 
closed geodesic.
Two possibilities arise:

\begin{enumerate}
\item[{\bf Case A:}]
 The geodesic realizations of $\gamma_{mn}$ lie on $P$ for infinitely many pairs $m, n$. 
\item[{\bf Case B:}]
The geodesic realizations of $\gamma_{mn}$ do not lie on $P$ for all but finitely many pairs $m, n$. 
\end{enumerate}

\subsection{Proof of Main Theorem in Case A} We lift the points $z_n$ to a totally geodesic plane $\tP$ in $\tM$. Also let $\til{\sigma}$
be the lift of the horocycle on which the points $\{ z_n \}$ lie.
Then there exist (at least) four points $a, b, c, d$ on $\til{\sigma}$ such that the geodesics $(a,b), (c,d)$ lie very close to 
lifts of the closed geodesics
$\gamma_1, \gamma_2 \in \{ \gamma_{mn} \}$. Extend $(a,b), (c,d)$ infinitely in both directions to get
bi-infinite geodesics $(a_\infty,b_\infty), (c_\infty,d_\infty)$. (See Diagram below.)

\smallskip

\begin{center}

\includegraphics[height=4cm]{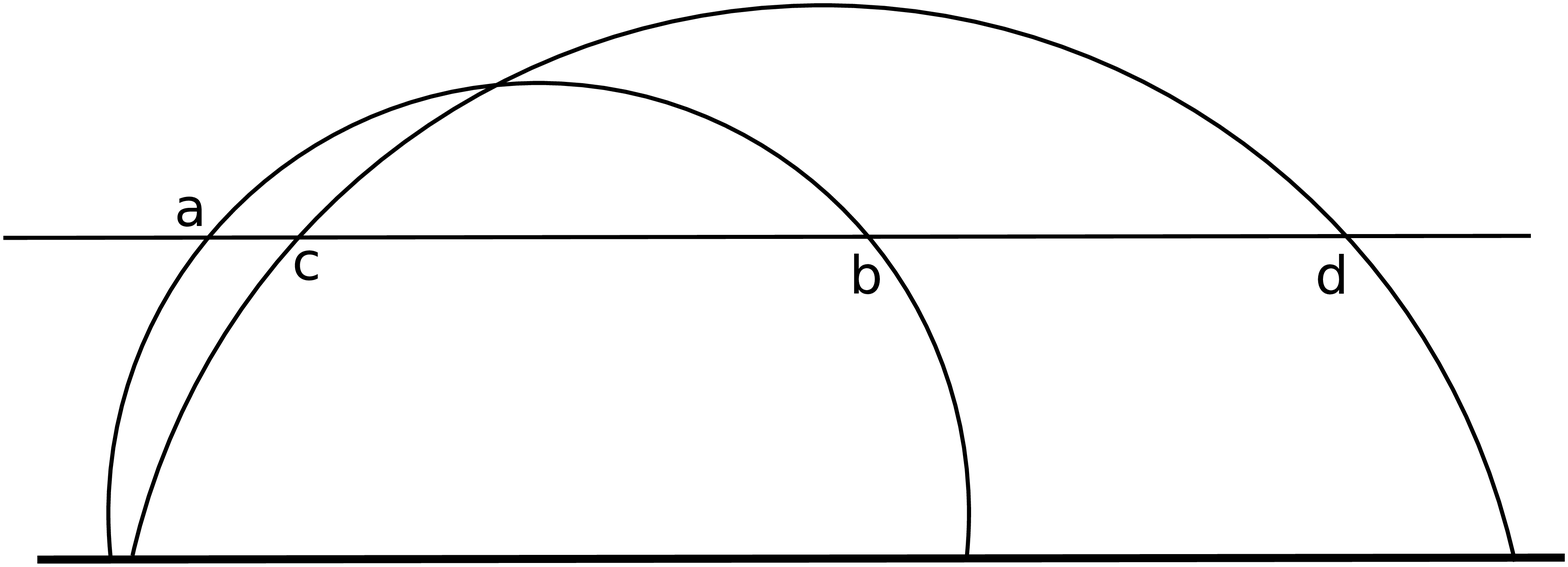}

\smallskip

\end{center}

\smallskip

 Since the collection $\gamma_{mn}$ is infinite, we can choose $a, b, c, d$ and a base-point $o \in \tP$ in such a way that 
\begin{enumerate}
\item the visual angles between the pairs $(a_\infty, \gamma_1 (-\infty))$, $(b_\infty, \gamma_1 (+\infty))$,
$(c_\infty, \gamma_2 (-\infty))$, $(d_\infty, \gamma_2 (+\infty))$ are all small. 
\item the distance between $b, d$ is large.
\end{enumerate}
It follows that at least three of points $\gamma_1 (-\infty), \gamma_1 (+\infty),  \gamma_2 (-\infty), \gamma_2 (+\infty)$
are distinct points on the boundary $S^2$. Since $\Ga$ is discrete without parabolics, it follows that all four points
$\gamma_1 (-\infty), \gamma_1 (+\infty),  \gamma_2 (-\infty), \gamma_2 (+\infty)$ are distinct and hence the group generated by 
$\gamma_1, \gamma_2$ is non-elementary. In other words, the stabilizer of $P$ is non-elementary. If $P$ is already
a closed immersed surface in $M$, then there is nothing left to prove for Theorem \ref{main}. Else, choose $P_1 \neq P$ such that $P_1$
is an immersed plane in the minimal set $X$. Then $P \subset \bbar{P_1}$ and the hypotheses for Theorem  \ref{zdcor} are satisfied.
It follows that $\bbar{\Ga (\partial \tP_1)} = \SSS$; or equivalently $X = M$.

\subsection{Proof of Main Theorem in Case B}
Case B  implies that (for infinitely many pairs $z_n, z_m$)
the short geodesic segments $[z_n,z_m]$ necessarily join points $z_n, z_m$ that lie in different sheets of $B \cap P$ for some small ball
$B$ in $M$. 
Hence there exist small elements $g_j\in G\setminus H$ such that $g_j Y \cap Y \neq \emptyset$ (e.g. choose $g_{mn}$ such that $g_{mn} z_m = z_n$).
Since $\bbar{Uy} = Y$ for all $y \in Y$, this  immediately gives (cf. \cite[Lemma 9.5]{mmo})
\begin{lemma}  For any $ y \in Y$  there exists a sequence $g_n \to 1$  in $G\setminus H$ such
that $g_n y\in Y$.\label{smallg} \end{lemma}

By Lemma \ref{smallg}, there is a sequence $g_n \to id$
 in $G\setminus H$ such that $g_nY\cap Y \neq \emptyset$. Since $UY=Y$ is minimal, it follows that $g_nY = Y$.
 From $g_n$ we extract a sequence $f_n \in AV
\setminus \{ 1\}$ such that $f_n \to 1;
f_n \neq 1$ and $f_n Y =Y$.

If $g_n \in AN$ for infinitely many $n$, then since $g_nU Y = g_nY=Y$, we can post-multiply $g_n$
by the inverse of its  $U$ component to get a sequence $f_n \to 1$ with $f_n \in AV \setminus \{ 1\}$
satisfying $f_n Y = Y$. 

Else let $g_n \notin AN$ for all but finitely many $n$. Since $Ug_nU Y = Y$,   it follows from
 Theorem
\ref{av} that for any neighborhood $G_0$ of  $1 \in G$, there exist $u_n,   u_n' \in U$ such that
$u_ng_nu_n' \to f \in   G_0 \cap (AV \setminus {id})$. Since $Ug_nU Y = Y$, and $Y$ is closed,
we have $f Y = Y$. Further,
since  $G_0$ is arbitrary, there exists a sequence $f_m \to 1$ in $(AV \setminus {id})$
satisfying $f_m Y = Y$.

Thus, in either case, there exists a sequence $f_m \to 1$ in $(AV \setminus {id})$ such that
 every element of the cyclic group $<f_m>$ preserves $Y$, i.e. $<f_m>Y=Y$.
 Passing to a subsequence and taking a Hausdorff limit, we obtain a closed, 1-parameter group 
$ L \subset AV$ such that $ LY = Y$ . Since $X$ is $H-$minimal it follows that $LX=X$ since for every $\lambda \in L$
$Y (=\lambda Y) \subset \lambda X \cap X$ forcing $\lambda X=X$. 

The cosets $LHy \subset X$ give a non-constant, continuous family of circles  whose union
contains an open subset of $S^2$. Hence, by Theorem \ref{openfull}
the collection of circles $\{ \partial P: P \in X\}$ is all of $\SSS$. Hence $X = M$.
This completes the proof of Theorem \ref{main}. \hfill $\Box$

\subsection{ Remarks on the compressible core case} The proof of Theorem \ref{main}  deals with
Alternative A of 
Proposition \ref{alter}, where a boundary component $S$ of the compact core $K$ carries an algebraic lamination essentially given
by $X \cap K$. As Lemma \ref{nonempty} points out, incompressibility of $K$ is sufficient to guarantee this alternative. 
It seems to us that Alternative B will require quite different techniques to handle. The test-case is when $\Ga$ is a degenerate
handlebody group, i.e. it is free without parabolics and has $S^2$ as its limit set. 

\section{Compact totally geodesic surfaces}\label{fin}

 For the purposes of this Section, we relax the assumption that $\Ga$ 
is degenerate and that $M$ has incompressible core. However, we do assume that $\Ga$ has no parabolics
and that it has a compact core $K$. In other words the main Theorem of this section applies to finitely generated
Kleinian groups without parabolics. Let $\Lambda_\Ga$ denote the limit set of $\Ga$.
We first recall some material from \cite{mahan-lecuire} to which we refer for details.

\subsection{Almost minimizing geodesics}
\begin{defn}

A geodesic $\gamma = \gamma(t): t \in \reals$ in $M$  is called {\bf 
almost minimizing} if it is has unit speed  and if there exist $C \geq 0$ 
such that  $ d_M(\gamma(0),\gamma(s)) \geq |s|- C$.

\end{defn}

It is easy to construct almost minimizing geodesic rays. 
Choose a sequence $z_n$ exiting 
an end $E$, join them to $S (= \partial E)$ by minimizing geodesic segments, and take a limit.
Any limiting ray is an almost minimizing geodesic ray.
It follows from work of Ledrappier \cite[Proposition 4]{ledrappier} and Eberlein \cite[Proposition 5.6]{eberlein2}
that the set of landing points in $S^2_\infty$ of (lifts to $\tM$ of) almost minimizing geodesic rays
coincides with  the complement of the {\it horospherical limit set} $\Lambda_h (\Ga)$
of $\Ga$. We shall not need this in what follows; however we shall denote the collection of landing points
of almost minimizing geodesic rays by $\Lambda_h (\Ga)^c$. 

Let $\partial i$ denote the (boundary value of)
the Cannon-Thurston map from the Gromov boundary $\partial \Ga$ to its limit set $\Lambda_\Ga$.

\begin{defn} The {\bf multiple limit set}
 $\Lambda_m (\Ga) = \{ x \in \Lambda_\Ga : \# |(\partial i)^{-1}(x)| > 1\}$. 
\end{defn}

Equivalently, by Theorem \ref{ctstr}, 

\begin{center}

$\Lambda_m = \{ \partial i (y) : y$ is an end-point of a leaf of   $\LL_E$ for some ending lamination $\LL_E \}.$

\end{center}

In \cite[Section 3]{mahan-lecuire}, we establish
\begin{prop}\cite{mahan-lecuire} $\Lambda_h (\Ga)^c = \Lambda_m$. 
 \label{horo-mult}
\end{prop}

We include a proof-sketch for completeness:\\

Fix a degenerate end $E$ of $M$.
To show that $ \Lambda_H^c \subset \Lambda_m$ it suffices to show that if $\til r$
is the lift of an almost minimizing geodesic $r$ to $\til E$, then it lands
on $\Lambda_m$. Choose a sequence of surfaces $S_i$ exiting the end $E$. Choose homotopically
essential simple closed curves 
$s_i \subset S_i$ of bounded length (bounded by $C_0$ say) such that $r \cap s_i \neq \emptyset$  
for all $i$. Choose simple closed curves $a_i$ on $S_0 = \partial E$ such that $r(0) \in a_i$
for all $i$ (we are just freely homotoping $s_i$ down to $S_0$ making sure that the curve passes
through $r(0)$). Consider the annulus with boundary components $a_i, s_i$ and containing the segment
$r_i$ of
$r$ from $r(0)$ to $r \cap s_i$. Lifting to the universal cover we obtain a quadrilateral 
whose `top edge' (corresponding to $s_i$) has length bounded by $C_0$ and whose `bottom edge'
(corresponding to $a_i$) has length tending to infinity as $i \to \infty$. Let $\til{r_i}$ 
(resp. ${\til{r_i}}^o$) be the lift
of $r_i$ joining the beginning points of $a_i, s_i$ (resp. the  lift
of $r_i$ with orientation reversed joining the end points of $s_i, a_i$). It follows that 
\begin{enumerate}
\item The lifts of curves $a_i$ to $\til M$ converge to a leaf $l$ of the ending lamination for $E$.
\item The Cannon-Thurston map identifies $l(-\infty), l(\infty)$ to $\til{r}(\infty)$.
\end{enumerate}
Thus $\til{r}(\infty) \in \Lambda_m$.

Conversely, if $z \in \Lambda_m$, then by Theorem \ref{ctstr}, there is a 
 an end $E$ and a bi-infinite leaf $(a, b)$ of the ending lamination $\LL_E$ corresponding to the end $E$
such that $\hat{i} (a) = \hat{i} (b) =z$. Then \cite[Chapter 9]{thurstonnotes} there exists
 a sequence $a_n$ of simple closed geodesics on $S (=\partial E)$ 
such that 
\begin{enumerate}
\item $a_n^{\pm \infty}$ converges to $\{ a, b \}$ (where $a_n^{\pm \infty}$ denote the attracting and repelling fixed points of $a_n$ acting on $\partial G$, the Gromov boundary of $G$). 
\item The geodesic realizations $s_n$ of $a_n$ in $E$ exit $E$.
\end{enumerate}
Joining $s_n$ to $a_n$ by a geodesic $r_n$ that realizes $d_M (s_n, a_n)$ and taking limits as $n\to \infty$,
 we see that $r_n$ converges (up to subsequencing) to an almost minimizing geodesic ray $r$
in $E$, such that
$r(\infty) = \hat{i} (a) = \hat{i} (b)$. This forces $r(\infty) = z$ and hence $z \in  \Lambda_H^c$. 
$\Box$

\subsection{Finitely many closed surfaces}
The rest of this section is devoted to showing:

\begin{theorem} \label{main2} Let $\Ga$ be a finitely generated Kleinian group without parabolics and let
$M = \hyps/\Gamma$. If $M$ has infinite volume, then there can exist only finitely many compact
totally geodesic surfaces in $M$. \end{theorem}

When $\Ga$ is convex cocompact, Theorem \ref{main2} is a direct consequence of \cite[Corollary B.2]{mmo}.
Therefore we are interested in the case when at least one end $E$ of $M$ is degenerate.
We argue by contradiction. Suppose that there is an infinite sequence $\{ S_i \}$ of totally geodesic surfaces
in $M$.
As an  immediate
consequence of Theorem \ref{bddplane}, we  have the
following:

\begin{lemma}\label{arbdeep}
Let $\Ga$ be a finitely generated Kleinian group without parabolics and let
$M = \hyps/\Gamma$. Suppose $M$ has infinite volume.
 If there is an infinite sequence $\{ S_i \}$ of totally geodesic surfaces
in $M$, then there exists an end $E$ such that (after passing to a subsequence if necessary)
 the sequence $\{ S_i \}$ satisfies the following:\\
 For any compact subset $Q$ of $M$, $S_i \cap (E \setminus Q)= \emptyset$ for only finitely many $i$.
\end{lemma}

\begin{proof}  
Since the fundamental groups $\pi_1(S_i)$ inject into $\pi_1(M)$, the limit sets of 
$\pi_1(S_i)$ are contained in $\Lambda_\Ga$, the limit set of $\Ga$. Hence the totally geodesic 
surfaces $\{ S_i \}$ are all contained in the convex core $CC(M)$ of $M$. By Theorem \ref{bddplane},
there can be only finitely many of the $\{ S_i \}$'s contained in any compact subset of $CC(M)$.
Since $CC(M)$ has finitely many ends, the Lemma follows by passing to a subsequence if necessary.
\end{proof}

In other words, the sequence  $\{ S_i \}$  penetrates arbitrarily deep into $E$. Let $S = bdy(E)$ be the boundary of the end $E$. 
 It follows, by Lemma \ref{intersect},
 that we may choose $S= K \cap E$, where $K$ is a compact core of $M$. 
Since $\Ga$ has no parabolics, it is necessarily a Gromov-hyperbolic group. Also, it follows from
Thurston's hyperbolization of atoroidal manifolds with boundary (see \cite{BF},
\cite[Theorem 4.6]{mitra-ht}  for a proof in the context of hyperbolic groups) that

\begin{lemma} Let $j: S \to M$ denote the inclusion map. Then $j_\ast( \pi_1(S) ) = \Delta$ is quasiconvex in $\Ga$.
\label{periqc} \end{lemma}

Let $\Delta_i$ denote the subgroup of $\Ga$ corresponding to $\pi_1(S_i)$. Since $S_i$ is totally geodesic,
$\Delta_i$ is quasiconvex in $\Ga$. Since the intersection of quasiconvex groups is quasiconvex 
\cite[Proposition 3]{short}, it follows that $\GG_i := \Delta \cap \Delta_i$ is quasiconvex in $\Ga$. 
We assume  (after perturbing $K$ slightly if necessary), that $S$ is transverse to $S_i$ for all $i$. Also each 
$S_i$ necessarily intersects $S$ (as we have chosen the sequence this way). 

Let $M_i$ be the cover of $M$ corresponding to the subgroup $\GG_i$.  Let $\bbar{S_i}$ and $\Sigma_i$
denote the unique lifts of $S_i, S$ respectively to $M_i$, such that $\pi_1(\bbar{S_i}) = \pi_1 (\Sigma_i) = \GG_i$.
Note that $\bbar{S_i}$ and $\Sigma_i$ are embedded submanifolds of $M_i$, and that $M_i$ is homeomorphic to
$\bbar{S_i} \times \R$ as well as $\Sigma_i \times \R$. We shall be interested in a compact core
of $M_i$ bounded by (pieces of) $\bbar{S_i}$ and $\Sigma_i$. Thus
 there exist
\begin{enumerate}
\item  A compact submanifold with boundary ${\bbar{S_i}}^0$ of $\bbar{S_i}$ whose inclusion into
 $\bbar{S_i}$ is a homotopy equivalence.
\item A compact submanifold with boundary ${\Sigma_i}^0$ of $\Sigma_i$ whose inclusion into
 $\Sigma_i$ is a homotopy equivalence.
\item ${\bbar{S_i}}^0 \cap {\Sigma_i}^0$ is a finite union of circles containing
$ bdy ({\bbar{S_i}}^0) = bdy ({\Sigma_i}^0)$. 
\end{enumerate}

Then there exist compact product regions in $M$ whose boundaries consist of (isotopic) submanifolds
of ${\bbar{S_i}}^0$ and ${\Sigma_i}^0$ intersecting only along their boundary circles. We shall 
consider one of these product regions $Q_i$. Passing to the cover of $M_i$ corresponding to $\pi_1(Q_i)$
and abusing notation
slightly, we  set $\GG_i = \pi_1(Q_i)$ and assume that  the
compact manifold $Q_i$ has boundary given by $bdy(Q_i)={\bbar{S_i}}^0 \cup {\Sigma_i}^0$.  
Also, the inclusions of ${\bbar{S_i}}^0$ or ${\Sigma_i}^0$ into $Q_i$ is a homotopy equivalence, as is the
inclusion of $Q_i$ into $M_i$. Thus, by passing to a subsurface if necessary, we are assuming that 
${\bbar{S_i}}^0 \cap {\Sigma_i}^0$ is exactly equal to
$ bdy ({\bbar{S_i}}^0) = bdy ({\Sigma_i}^0)$; and further that ${\bbar{S_i}}^0 \cup {\Sigma_i}^0$
bounds the product region $Q_i$. 

Let $\Pi_i: M_i \to M$ be the covering projection and
let $E_i$ be the lift of $E$ containing $\Sigma_i$.
Since the sequence of totally geodesic surfaces 
$\{ S_i \}$  penetrates arbitrarily deep into $E$, it follows that there exist points
$q_i \in \bbar{S_i}^0$ that realize the maximum distance (amongst points in $\bbar{S_i} \cap E_i$) from $\Sigma_i$.
Then after projecting back to $M$ using $\Pi_i$, $\Pi_i(q_i)$ is a point in $S_i\cap E$ that maximizes
(amongst points in ${S_i} \cap E$)
distance from $S$.

Let 
$p_i \in {\Sigma_i}$ be a point on ${\Sigma_i}$ closest to $p_i$, i.e. 
$d_i(p_i, q_i)=d_i(\Sigma_i, q_i)$, where $d_i$ denotes the
hyperbolic metric
on $M_i$.
Then $d_i(p_i, q_i) \to \infty$ as $i \to \infty$. Let $[p_i, q_i]$ be the shortest path between $p_i, q_i$.
Then $[p_i, q_i]$ is perpendicular to $\bbar{S_i}$ at $q_i$ (since $q_i$ maximizes distance).
We note further, that since $[p_i, q_i]$ is a minimizing geodesic segment, it remains a minimizing geodesic segment
after projecting back to $M$. Since a limit of minimizing geodesic segments is an almost
minimizing geodesic ray, we have the following.

\begin{lemma} \label{am} 
Any limit (as $i \to \infty$) of the sequence 
$\{ \Pi_i ([p_i, q_i]) \}$ is an almost minimizing geodesic ray. \end{lemma}

We shall also need the following:

\begin{lemma} \label{concat} Let $\AAA_i$ be the family of geodesic rays in 
$\bbar{S_i}$ starting at $q_i$ and
let $a_i \in \AAA_i$ be one of these rays. Then
the concatenation $\alpha_i=[p_i, q_i]\cup a_i$ lifted to the universal cover $\tM$ is a uniform
(independent of $a_i \in \AAA_i$)
quasigeodesic in $\tM$. 
\end{lemma} 

\begin{proof}
This follows from the well-known fact that the concatenation of two geodesics perpendicular
to each other is a uniform quasigeodesic (see \cite[Lemma 3.3]{mitra-trees} for instance).  \end{proof}

Consider the family of geodesic segments $\RR_i$ 
(in the intrinsic metric) on ${\Sigma_i}$ starting at $p_i$ and ending on 
$ bdy ({\Sigma_i}^0) (=bdy ({\bbar{S_i}}^0))$. Then the concatenation $\gamma_i=[q_i, p_i]\cup r_i$ for $r_i\in \RR_i$
is path-homotopic to a geodesic segment contained in a unique $a_i \in \AAA$. Thus, starting with
$r_i\in \RR_i$ we perform the following operations:
\begin{enumerate}
\item First adjoin $[q_i, p_i]$ to its beginning and path-homotop it to a  geodesic subsegment of
a unique
$a_i \in \AAA_i$.
\item Then  adjoin $[p_i, q_i]$ to the beginning of the subsegment of $a_i$ thus obtained
and get a uniform quasigeodesic segment
$\alpha_i$ starting at $p_i$.
\item
Note that $\alpha_i$ is
 path homotopic in $M_i$ to the original $r_i\in \RR_i$ we started with.
\end{enumerate}

Next, lift everything to the universal cover $\tM$ and assume that all the $p_i$'s are lifted to lie in a fixed fundamental domain in $\tS$ (a lift of $S$ to $\tM$). Let $\RR_\infty$ be the family  of infinite geodesic rays
$r_\infty$ in $\tS$ that
 satisfy the following property: \\
There exists a sequence of geodesic segments $\{ r_i (\in \RR_i) \}$ such that
$r_i \subset r_\infty$ for all $i$ and $\cup_i r_i = r_\infty$.

Let $\partial \RR_\infty $ denote the collection of landing points of these rays in $\partial \Gamma$
(the Gromov boundary of $\Ga$). Recall that $j_\ast( \pi_1(S) ) = \Delta$ is quasiconvex in $\Ga$
by Lemma \ref{periqc}.
Clearly, $\partial \RR_\infty  \subset \partial \Delta$. Also if the rays $\RR_i$ are extended infinitely
they would land on $\partial \Delta \setminus \partial \GG_i$. Hence $\partial \Delta 
\setminus (\bigcup_i \partial \GG_i) \subset \partial \RR_\infty $. Note that
since $\GG_i$ is quasiconvex in $\Delta$, it follows that $ \partial \GG_i$ is nowhere dense in $\partial \Delta$
and also has zero measure (with respect to any visual measure). 
We have shown

\begin{lemma} \label{fullmre}
 $\partial \Delta \setminus
\partial \RR_\infty$
is nowhere dense in $\partial \Delta$ and  has zero measure with respect to any visual measure on it.  
\end{lemma}

Recall that
 any limit (as $i \to \infty$) of $\Pi_i([p_i, q_i])$'s is almost minimizing by Lemma \ref{am}. Since 
$\alpha_i=[p_i, q_i]\cup a_i$ is a uniform quasigeodesic by Lemma \ref{concat}, it follows that all limits of
the $\alpha_i$'s (as $i \to \infty$) are also almost minimizing. 
Let $\BB$ denote the set of limiting rays of the $\{ \alpha_i \}$'s. Let $\partial \BB$ denote the
set of landing points in $S^2_\infty$ of rays in $\BB$.

Note now that each  ray
$r_\infty \in \RR_\infty \subset \tS$ is a limit of $r_i\in \RR_i$. Let $r_\infty (\infty)$
denote its landing point in the Gromov boundary $\partial \Delta \subset \partial \Ga$. 
We have seen above that there exists
a uniform quasigeodesic $\alpha_i$ in $M_i$ path-homotopic to $r_i$. Since a
Cannon-Thurston map $\partial i$ exists by Theorem \ref{ctstr}, it follows that 
$\partial i (r_\infty (\infty))$ belongs to $ \partial \BB$ for all $r_\infty \in \RR_\infty$.
By Proposition \ref{horo-mult}, it follows that $\partial i (r_\infty (\infty)) \in \Lambda_m$, the multiple
limit set, for all $r_\infty \in \RR_\infty$. Hence by Theorem \ref{ctstr}, $r_\infty (\infty) \in \partial \Delta$
is an ideal end-point of a leaf of the ending lamination $\LL_E$ corresponding to the end $E$. 
The set of
all such end-points is of zero measure and is nowhere dense (see 
\cite[Ch. 8]{thurstonnotes} and \cite{soma} for instance). This contradicts
Lemma \ref{fullmre} and proves Theorem \ref{main2}. \hfill $\Box$

\section*{Acknowledgments} I am extremely grateful to Anish Ghosh for several exciting and instructive conversations.
Special thanks are due to Curt McMullen for raising the question \cite{ctm-pc} that Theorem \ref{main2} answers;
and also for  insightful comments on the Ratner-type phenomenon explicated in \cite{mmo}.
\bibliography{degplan}
\bibliographystyle{alpha}

\end{document}